\documentclass[11pt]{article}
\usepackage{amsmath,amsthm,amsfonts,amssymb,amscd, amsxtra}
\usepackage{graphicx,multirow}
\usepackage{graphicx,color}
\usepackage{float}
\usepackage[margin=2.5cm,nohead]{geometry}
\newtheorem{theorem}{Theorem}[section]

\newtheorem{lemma}{Lemma}[section]
\newtheorem{proposition}{Proposition}[section]
\newtheorem{example}{Example}[section]
\newtheorem{remark}{Remark}[section]

\newtheorem{definition}{Definition}[section]

\def\min{\operatorname{min}}

\DeclareMathOperator{\grad}{grad}

\def\min{\operatorname{Minimize}}
\def\const{\operatorname{subject~to~}}
\title{Enlargement of Monotone Vector Fields and an Inexact Proximal  Point Method for Variational Inequalities  in Hadamard Manifolds.}
\author{Batista,  E. E. A.\thanks{IME, Universidade Federal do Oeste da Bahia,
Barreiras, BA 47808-021, BR ({\tt edvaldo.batista@ufob.edu.br}).}
\and  Bento, G. C. \thanks{IME, Universidade Federal de Goi\'as,
Goi\^ania, GO 74001-970, BR ({\tt glaydston@ufg.br}).}
\and
 Ferreira, O. P.
\thanks{IME, Universidade Federal de Goi\'as,
Goi\^ania, GO 74001-970, BR({\tt orizon@ufg.br}).}
}
\begin{document}
%%%%%%%%%%%%%%%%%%%%%%%%%%%%%%%%%%%%%%%%%%%%%%%%%%%%%%%%%%%%%
\maketitle
\begin{abstract}

In this paper  an inexact proximal  point method for variational inequalities  in  Hadamard manifolds is introduced and studied its convergence properties. The main tool  used for presenting the method is the concept of  enlargement of monotone vector fields, which generalizes   the  concept of enlargement of monotone operators from the linear setting  to the  Riemannian context.   As an application,  an inexact proximal point method  for  constrained optimization problems  is obtained.\\


\end{abstract}
%%%%%%%%%%%%%%%%%%%%%%%%%%%%%%%%%%%%%%%%%%%%%%%%%%%%%%%%
\section{Introduction}
%%%%%%%%%%%%%%%%%%%%%%%%%%%%%%%%%%%%%%%%%%%%%%%%%%%%%%%%
In the last few  years,  there has been increasing the number of papers dealing with the subject of the extension of concepts and techniques,  as well as methods of mathematical programming,    from  the  linear setting   to  the Riemannian context;  papers published in the last three years about  this issues include, for example,    \cite{Absil2014,  AhmadiKhatibzadeh2014, AABCP2014-3, Bacak2013,  BentoCruzNeto2013, BCNS2013,   BFO2012,  BentoFerreira2015, BonnelUd2015,  HosseiniPouryayevali2013,  Manton2015,  PapaQuirozOliveira2012, WangLi2015, WangLiLopezYao2015, WangLiYao2015,    WenWotao2013}. Is well known that  convexity and monotonicity  plays an important role in the analysis and development of methods of mathematical programming. Hence,  one of the reasons for this  extension is the possibility to transform non-convex or non-monotone problems in Euclidean context into  Riemannian convex or monotone  problems, by introducing a suitable metric, which  allow modify  numerical  methods to find  solutions of these problems;  see \cite{ BentoFerreira2015, BentoMelo2012, CLMM2012, FerreiraCPN2006, Rapcsak1997}.  These extensions,  which in general are nontrivial,   are either of purely theoretical nature or aims at obtaining  numerical  algorithms.  Indeed, many mathematical programming problems are naturally posed on Riemannian manifolds having  specific underlying geometric and algebraic structure that could be also exploited to reduce the cost of obtaining the solutions; see, e.g., \cite{ Absil2014,  AdlerDedieuShub2002, EdelmanAriasSmith1999, Karmarkar1998,  Manton2015, MillerMalick2005,  NesterovTodd2002,  Smith1994, WenWotao2013}.

In this paper, we consider the problem of  finding a  solution of a variational inequality  problem  defined on a Riemannian manifold. Variational inequality problems on Riemannian manifolds were first introduced and studied  by  N\'emeth in \cite{Nemeth2003} for univalued vector fields on Hadamard manifolds  and  for multvalued  vector fields on general Riemannian manifolds  by Li and Yao in   \cite{LiYao2012};  for recent works addressing this subject  see \cite{FangChen2015, LiChongLiouYao2009, TangWangLiu2015, TangZhouHuang2013}. It is worth to point out that  constrained optimization problems and the problem of finding the zero of a  multivalued vector field,   studied  in \cite{AhmadiKhatibzadeh2014, BentoFerreira2015,    daCruzFerreiraPerez2006,  FerreiraOliveira2002, LiLopesMartin-Marquez2009, WangLiLopezYao2015},  are  particular instances of the variational inequality  problem.

The aim  of this   paper is to  present an inexact proximal  point method for variational inequalities  in Hadamard manifolds and to study its convergence properties. As an application, we obtain an inexact proximal  point method for constrained optimization problem in Hadamard manifolds.  In order to present our method, we first generalize the concept of enlargement of monotone operators,  introduced  by \cite{BurachikIusemSvaiter1997}, from linear setting to the Riemannian context; see also \cite{BurachikIusem1998}. It is worth mentioning that the concept of enlargement of monotone operators  in linear spaces has been successfully employed for wide range of purpose; see \cite{BurachikIusem2008} and its reference therein. As far as we know, this is the first time that the  inexact proximal  point method for variational inequalities is studied in the Riemannian setting.   Finally, we also  mention that  the method  introduced   has two  important particular instances, namely,   the methods  (5.1) of \cite{LiYao2012} and  (4.3) of \cite{LiLopesMartin-Marquez2009}.

The organization of the paper is as follows. In Section~\ref{sec:aux},  some notations and  basic results used in the paper are presented. In Section~\ref{sec2}, the  concept of enlargement of monotone vector fields is introduced and  some properties are obtained. In Section \ref{sec3} the inexact proximal point method for variational inequalities is presented  and studied its convergence properties. As an application, in Section~\ref{sec4} an inexact proximal point method  for  constrained optimization problems  is obtained.  Some final remarks are made in Section~\ref{secfr}.
%%%%%%%%%%%%%%%%%%%%%%%%%%%%
\subsection{Notation and Terminology} \label{sec:aux}
%%%%%%%%%%%%%%%%%%%%%%%%%%%%
In this section, we introduce some fundamental  properties and notations about Riemannian geometry. These basics facts can be found in any introductory book on Riemannian geometry, such as in \cite{doCarmo1992} and \cite{Sakai1996}.

Let $M$ be a $n$-dimentional Hadamard  manifold. {\it In this paper, all manifolds $M$ are assumed to be Hadamard finite dimensional}. We denote by $T_pM$ the $n$-dimentional {\it tangent space} of $M$ at $p$, by $TM=\cup_{p\in M}T_pM$ {\itshape{tangent bundle}} of $M$ and by ${\cal X}(M)$ the space of smooth vector fields on $M$. The Riemannian metric is denoted by  $\langle \,,\, \rangle$ and  the corresponding norm  by $\| \; \|$. Denote the lenght of piecewise smooth curves $\gamma:[a,b]\rightarrow M$ joining $p$ to $q$, i.e., such that $\gamma(a)=p$ and $\gamma(b)=q$, by
\[
l(\gamma)=\int_a^b\|\gamma^{\prime}(t)\|dt,
\]
and the Riemannian  distance by $d(p,q)$,  which induces the original topology on $M$, namely,  $( M, d)$ is a complete metric space and bounded and closed subsets are compact.  For $A\subset M$,  the notation  $ \mbox{int} (A)$ means the interior of the set $A$, and if $A$ is a nonempty set,  the distance  from $p\in M$ to $A$ is given by $d(p,A):= \inf \{d(p,q)~:~ q\in A\}.$ The metric induces a map $f\mapsto\grad f\in{\cal X}(M)$ which associates to each function smooth over $M$ its gradient via the rule $\langle\grad f,X\rangle=d f(X),\ X\in{\cal X}(M)$.   Let $\nabla$ be the Levi-Civita connection associated to $(M,{\langle} \,,\, {\rangle})$. A vector field $V$ along $\gamma$ is said to be {\it parallel} if $\nabla_{\gamma^{\prime}} V=0$. If $\gamma^{\prime}$ itself is parallel we say that $\gamma$ is a {\it geodesic}. Given that geodesic equation $\nabla_{\ \gamma^{\prime}} \gamma^{\prime}=0$ is a second order nonlinear ordinary differential equation, then geodesic $\gamma=\gamma _{v}(.,p)$ is determined by its position $p$ and velocity $v$ at $p$. It is easy to check that $\|\gamma ^{\prime}\|$ is constant. We say that $ \gamma $ is {\it normalized} if $\| \gamma ^{\prime}\|=1$. The restriction of a geodesic to a  closed bounded interval is called a {\it geodesic segment}. Since $M$ is a Hadamard manifolds the lenght   of the  geodesic segment  $\gamma$  joining $p$ to $q$ its equals $d(p,q)$, the parallel transport along $\gamma$ from $p$ to $q$ is denoted by $P_{pq}:T_{p}M\to T_{q}M$. Moreover, {\it exponential map} $exp_{p}:T_{p}  M \to M $ is defined by $exp_{p}v\,=\, \gamma _{v}(1,p)$ is  a diffeomorphism and, consequently, $M$ is diffeomorphic to the Euclidean space $\mathbb{R}^n $, $ n=dim M $.  Let ${q}\in M $ and $exp^{-1}_{q}:M\to T_{p}M$ be the inverse of the exponential map. Note that $d({q}\, , \, p)\,=\,||exp^{-1}_{p}q||$,  the map $d_{q}^2: M\to\mathbb{R}$ defined by $ d_{q}^2(p)=d(q,p)$ is  $C^{\infty}$ and
\begin{equation} \label{eq:gd2}
\grad d_{q}^2(p):=-2exp^{-1}_{p}{q}.
\end{equation}
 Furthermore,  we know that
\begin{equation} \label{eq:coslaw}
d^2(p_1,p_3)+d^2(p_3,p_2)-2\langle \exp_{p_3}^{-1}p_1,\exp_{p_3}^{-1}p_2\rangle\leq d^2(p_1,p_2),  \qquad p_1, p_2 , p_3 \in M.
\end{equation}
\begin{equation} \label{eq:coslaw2}
\langle \exp^{-1}_{p_2}p_1, \, exp^{-1}_{p_2}p_3\rangle+\langle \exp^{-1}_{p_3}p_1, \, exp^{-1}_{p_3}p_2\rangle \geq d^2(p_2,p_3),  \ \qquad p_1, p_2 , p_3 \in M.
\end{equation}
A set,  $\Omega\subseteq M$ is said to be {\it convex}  if any geodesic segment with end points in $\Omega$ is contained in
$\Omega$, that is,  if $\gamma:[ a,b ]\to M$ is a geodesic such that $x =\gamma(a)\in \Omega$ and $y =\gamma(b)\in \Omega$; then $\gamma((1-t)a + tb)\in \Omega$ for all $t\in [0,1]$.   Given an arbitrary set,  $\mathcal{B} \subset M$,   the minimal convex subset that contains  $\mathcal{B}$  is called  the {\it convex hull} of $\mathcal{B}$ and is  denoted by $ \mbox{conv}(\mathcal{B})$; see \cite{CLMM2012}.  Let $\Omega\subset\mathbb{R}^n$ be a convex set, and $p\in \Omega$. Following \cite{LiLopesMartin-Marquez2009}, we define the {\it normal cone} to $\Omega$ at $p$ by
\begin{equation} \label{eq:nc}
N_{\Omega}(p):=\left\{w\in T_pM~:~\langle w, \exp_{p}^{-1}q \rangle\leq 0, q\in \Omega \right\}.
\end{equation}
Let $f:M\to\mathbb{R}\cup\{+\infty\}$  be a function. The {\it domain} of $f$ is the set defined by
\begin{eqnarray*}
\mbox{dom}f:=\left\{ p\in M ~: ~f(p)<\infty \right\}.
\end{eqnarray*}
The function $f$ is said to be proper if $\mbox{dom}~f\neq \varnothing$ and  {\it convex} on a convex set $\Omega\subset \mbox{dom}~f$ if for any geodesic segment $\gamma:[a, b]\to\Omega$ the composition $f\circ\gamma:[a, b]\to\mathbb{R}$ is convex.  Is very known that  $d_{q}^2$ is convex.  Take $p\in \mbox{dom}~f$. A vector $s \in T_pM$ is said to be a {\it subgradient\/} of $f$ at $p$, if
\[
f(q) \geq f(p) + \langle s, \, \exp^{-1}_pq\rangle,  \qquad  q\in M.
\]
The set  $\partial f(p)$ of all subgradients of $f$ at $p$ is called the {\it subdifferential\/} of $f$ at $p$. The function  $f$ is {\it lower semicontinuous} at $\bar{p}\in\mbox{dom}f$ if for each sequence $\{p^k\}$ converging to $\bar{p}$ we have
\begin{eqnarray*} \liminf_{k\rightarrow\infty} f(p^k)\geq f(\bar{p}).
\end{eqnarray*}
 Given a multivalued vector field  $X: M \rightrightarrows  TM$,   the  domain of $X$  is the set defined by
\begin{equation} \label{eq:dr}
\mbox{dom}X:=\left\{ p\in M ~: ~X(p)\neq \varnothing \right\},
\end{equation}
Let  $X: M \rightrightarrows  TM$ be a vector field and $\Omega\subset M$. We define the following quantity
 $$
 m_X(\Omega):=\sup_{p\in\Omega}\left\{\|u\|~:~u\in X(p)\right\}.
 $$
We say that  $X$  is {\it locally bounded} if,  for all $p \in  \mbox{int}( \mbox{dom}X) $,   there exist an open  set $U\subset M$  such  that   $p\in U$ and  there holds $m_X(U) < +\infty,$ and {\it bounded on bounded sets }  if for all bounded set $V\subset M$  such that its closure $\overline{V} \subset  \mbox{int}( \mbox{dom}X)$ it holds that   $m_X(V) < +\infty$.  The multivalued vector field   $X$ is said to be {\it upper semicontinuous} at $p\in\mbox{dom}X $  if, for any open set  $V\subset T_pM$ such that $X(p) \in V$, there exists an open set  $U \subset M$ with $p\in U$ such that $P_{qp}X(q)\subset V$, for any $q\in U$. For two  multivalued vector fields  $X, Y$ on $M$, the notation $ X\subset Y$ means $X(p)\subset Y(p)$, for all $p\in M$.

A sequence $\{p^k\} \subset (M,d)$ is said  to be  {\it  quasi-Fej\'er convergent} to a nonempty set $W\subset M$  if, for every $q \in W$ there exists a sommable sequence $\{\epsilon_k\}\subset\mathbb{\mathbb{R}_{++}}$,  such that $d^2(q, p^{k+1})\leq d^2(q, p^k) + \epsilon_k$, for $k=0,1, \ldots$. 

We end this section with a result, which its proof is analogous to the proof of Theorem 1 in Burachik et al. \cite{BurachikDIS1995}, by replacing the Euclidean distance by the Riemannian distance.
\begin{proposition}\label{fejer}
Let $\{p^k\}$ be a sequence in $(M,d)$. If $\{p^k\}$ is quasi-Fej\'er convergent to non-empty set $W\subset M$, then $\{p^k\}$ is bounded. If furthermore, an accumulation point $p$ of $\{p^k\}$ belongs to $W$, then $\lim_{k\rightarrow \infty}p^k=p$.  
\end{proposition}
%%%%%%%%%%%%%%%%%%%%%%%%%%%%%%%%%%%%%%
\section{Enlargement of Monotone Vector Fields} \label{sec2}
%%%%%%%%%%%%%%%%%%%%%%%%%%%%%%%%%%%%%%
A multivalued vector field $X$    is said to be {\it monotone}   if
\begin{equation}\label{eq2.1}
\left\langle P_{qp}^{-1} u-v, \, \exp_{q}^{-1}p\right\rangle \geq 0, \qquad  \qquad  ~ p,\,q\in \mbox{dom}X,  \quad u\in X(p), ~ v\in X(q),
\end{equation}
and  {\it strongly monotone}, if there exists $\rho>0$ such that
\begin{equation}\label{eq2.2}
\left\langle P_{qp}^{-1} u-v, \,\exp_{q}^{-1}p\right\rangle \geq \rho d^2(p,q), \qquad \qquad  ~ p,\,q\in \mbox{dom}X,  \quad u\in X(p), ~ v\in X(q).
\end{equation}
Moreover,  a monotone vector field $X$  is said to be  {\it maximal monotone\/}, if for each $p\in \mbox{dom}X$ and $u\in T_pM$, there holds:
\begin{equation}\label{eq2.3}
\left\langle P_{qp}^{-1} u-v, \,\exp_{q}^{-1}p\right\rangle \geq 0, \quad \qquad ~ q\in \mbox{dom}X, \quad ~ v\in X(q)  ~ \Rightarrow ~  u\in X(p).
\end{equation}
\begin{theorem}\label{mmsub}
 Let $f$ be a proper, lower semicontinuous and convex function on $M$. The subdifferential $\partial f$ is a monotone multivalued vector field. Furthermore, if $\emph{dom}f=M$, then the subdifferential $\partial f$ of $f$ is a maximal monotone vector field.
\end{theorem}
\begin{proof} See   \cite[Theorem 5.1]{LiLopesMartin-Marquez2009}.
\end{proof}
\begin{lemma} \label{le:msvf}
 Let $X_1,X_2$ be a maximal monotone vector fields such that $\emph{dom}X_1=\emph{dom}X_2=M$. Then $X_1+X_2$  is a maximal monotone vector field.
\end{lemma}
\begin{proof} Let $z\in M$. Define the following operator $  T_1,T_2: T_zM \rightrightarrows T_zM$ by
$$
T_1(u)=P_{exp_zu,z}X_1(exp_zu), \qquad \quad T_2(u)=P_{exp_zu,z}X_2(exp_zu),
$$
associated to   $X_1$ and $X_2$, respectively.  Since the parallel transport is linear,  then there holds
\begin{equation} \label{eq:st1t2}
(T_1+T_2) (u)=P_{exp_zu,z}(X_1+X_2)(exp_zu),  \qquad  u \in T_zM.
\end{equation}
Using that $X_1$ and $X_2$ are maximal monotone, then  it follows  from \cite[Theorem 3.7]{LiLopesMartin-Marquez2009} that  $T_1$ and $T_2$ are upper semicontinuous,  $T_1(u)$ and $T_2(u)$ are closed and convex for each $ u \in T_zM$. Thus, we conclude that  $T_1$ and $T_2$ are maximal monotone,  see \cite[Theorem 2.5, p. 155]{Cioranescu1990}.  Since   $T_1$ and $T_2$ are maximal monotone and $\mbox{dom}(T_1)=\mbox{dom}(T_2)=T_zM$, we conclude  from  \cite[Corollary 24.4 (i), p. 353]{BauschkeCombettes2011}  that $T_1+T_2$ is maximal monotone. Therefore, combining \eqref{eq:st1t2} with  \cite[Theorem 3.7]{LiLopesMartin-Marquez2009},  we conclude that $X_1+X_2$ is maximal monotone, which conclude the proof.
\end{proof}
\begin{lemma}\label{mon.cone}
Let $X$ be a maximal monotone vector field such that $\emph{dom}X=M$. Then $X+N_{\Omega}$  is a maximal monotone vector field.
\end{lemma}
\begin{proof}
The monotonicity of the $X+N_{\Omega}$ is immediate from the monotonicity of $X$ and definition of $N_{\Omega}$. Then, take $p\in M$  and  let $u\in T_pM$ be such that 
\begin{equation}
-\langle u,\exp^{-1}_pq\rangle-\langle v+w,\exp^{-1}_qp\rangle\geq0,\qquad   q\in M, ~ v\in X(q), ~w\in N_{\Omega}(q).
\end{equation}
Taking $w=0$ in last inequality  and using the maximality of $X$ we obtain that $u\in X(p)$ and therefore  $u+0\in (X+N_{\Omega})(p)$, which conclude the proof.
\end{proof}
\begin{proposition}\label{prop.sm}
Let $X$ be a multivalued  monotone vector field on $M$,  $q\in M$ and  $\lambda>0$. Then $X+\lambda\grad d_{q}^2$ is a strongly monotone vector field. Moreover, if   $X$ is maximal then $X+\lambda\grad d_{q}^2$ also maximal.
\end{proposition}
\begin{proof} The first part follows by combination of \eqref{eq2.1},  \eqref{eq2.2} and \cite[Proposition 3.2]{daCruzFerreiraPerez2002}. The second part follows by straight combination of  the convexity of $d_{q}^2$,  Theorem~\ref{mmsub} and Lemma~\ref{le:msvf}.
\end{proof}
Next, we define an operator that play an important rule in this paper.
\begin{definition} \label{def.enl.X}
Let  $X$  be a multivalued monotone vector field on $M$   and $\epsilon \geq 0$.   The enlarged vector field  $X^{\epsilon}: M   \rightrightarrows  TM $  associated to  $X$   is defined by
\begin{equation} \label{enl.X}
    X^{\epsilon}(p):=\left\{ u\in T_pM~:~ \left\langle P_{qp}^{-1} u-v, \,\exp_{q}^{-1}p\right\rangle \geq  -\epsilon, ~  q\in \emph{dom}X, ~  v\in X(q) \right\}, \qquad p\in \emph{dom}X.
\end{equation}
\end{definition}
\begin{example}\label{ex:xe}  Let  $\epsilon\geq 0$ and   $\bar{p}\in M$.   Define the closed ball at the origin $0_{T_pM}$ of $T_pM$ and radius $2\sqrt{2\epsilon}$ by 
 $$
 B\left[0_{T_pM}, ~2\sqrt{2\epsilon}\right]:=\left\{w\in T_pM~:~\parallel w\parallel \leq 2\sqrt{2\epsilon}\right\}.
 $$
 Denote the  enlarged vector field  of    $\partial d_{\bar{p}}^2(p)=\{ \grad d_{\bar{p}}^2(p)\}$ by   $\partial^{\epsilon} d_{\bar{p}}^2$.   We claim that the following inclusion holds 
 $$
 \partial d_{\bar{p}}^2(p) + B\left[0_{T_pM}, ~2\sqrt{2\epsilon}\right] \subseteq \partial^{\epsilon} d_{\bar{p}}^2(p),  \qquad   p\in M.
 $$
 Indeed,  first note that   from  \eqref{eq:gd2}  we conclude that   $\partial d_{\bar{p}}^2(q)=\{-2\exp^{-1}_q\bar{p}\}$,  for each $ q\in M$.  Due to $ \emph{dom} \partial d_{\bar{p}}^2=M$   definition of $ \partial^{\epsilon} d_{\bar{p}}^2$ implies 
  \begin{equation}\label{ex.1.i}
  \partial^{\epsilon} d_{\bar{p}}^2(p)=\left\{ u\in T_pM~:~ -\langle u, \, exp^{-1}_pq\rangle + \langle 2\exp^{-1}_q\bar{p}, \, exp^{-1}_qp\rangle \geq -\epsilon, ~  q\in M \right\}, \quad p\in M.
\end{equation}
We are going to prove the auxiliary result    $\{-2\exp^{-1}_p\bar{p}\}+A(p)\subset  \partial^{\epsilon} d_{\bar{p}}^2(p)$  for each $ p\in M$, where  \begin{equation}\label{eq:ap}
A(p)=\left\{w\in T_pM ~:~ 0 \geq -2d^2(p,q)+ \|w\|d(p,q)- \epsilon, ~ q\in M \right\}, \quad p\in M.
\end{equation}
First of all, note that by  using  \eqref{eq:coslaw2},  we obtain the  following inequality 
$$
2\left[\langle \exp^{-1}_p\bar{p}, \, exp^{-1}_pq\rangle+\langle \exp^{-1}_q\bar{p}, \, exp^{-1}_qp\rangle -d^2(p,q)\right] \geq 0,  \quad p, q \in M.
$$
Take  $w\in A(p)$.  Since $\langle w,exp^{-1}_pq\rangle\leq \|w\|d(p,q)$,  for all $w\in A(p)$ and $p, q\in M$, combining   \eqref{eq:ap} with last inequality yields 
$$
2\left[\langle \exp^{-1}_p\bar{p},exp^{-1}_pq\rangle+\langle \exp^{-1}_q\bar{p},exp^{-1}_qp\rangle -d^2(p,q)\right] \geq -2d^2(p,q)+\langle w,exp^{-1}_pq\rangle-\epsilon, \quad p, q \in M.
$$
Simple algebraic manipulations  in last inequality shows  that it is equivalent to the following ones
$$
-\langle -2\exp^{-1}_p\bar{p} +w, \, exp^{-1}_pq\rangle+\langle 2\exp^{-1}_q\bar{p}, \, exp^{-1}_qp\rangle  \geq -\epsilon, \quad p, q \in M, 
$$
which, from \eqref{ex.1.i},  allows to conclude that $-2\exp^{-1}_p\bar{p} +w \in  \partial^{\epsilon} d_{\bar{p}}^2(p)$, for all $w\in A(p)$ and $p\in M$.  Thus, the auxiliary result is proved. Finally, note that $w\in A(p)$ if,  and only if, there holds $\|w\|^2-8\epsilon <0, $ or equivalently, $\|w\|<2\sqrt{2\epsilon}$. Therefore, $A(p)= B\left[0_{T_pM}, ~2\sqrt{2\epsilon}\right]$ and,  because  $ \partial d_{\bar{p}}^2(p)+A(p)\subset  \partial^{\epsilon} d_{\bar{p}}^2(p)$  for each $ p\in M$,    the proof of the claim is done.  
\end{example}
\begin{remark}
Note that if $M$ has zero curvature then the inequality \eqref{eq:coslaw2} holds as a equality. Therefore,  
 in Example~\ref{ex:xe}, we can prove that  the inequality holds as equality, namely, 
$$
 \partial d_{\bar{p}}^2(p) + B\left[0_{T_pM}, ~2\sqrt{2\epsilon}\right] = \partial^{\epsilon} d_{\bar{p}}^2(p),  \qquad   p\in M. 
 $$
\end{remark}
\begin{proposition} \label{prop.elem.ii}
Let  $X$ be a monotone vector field on $M$ and $\epsilon\geq0$. Then,  $X\subset X^{\epsilon}$ and   $\emph{dom}X \subset \emph{dom}X^\epsilon$. In particular,  if $\mbox{dom}X=M$ then $\emph{dom} X^\epsilon=\emph{dom}X$.  Moreover, if $X$  is maximal then    $X^0=X$.
\end{proposition}
\begin{proof}
Take $\epsilon\geq0$. Since $X$ is monotone,  the first part of the proposition follows straightly from  \eqref{eq2.1} and \eqref{enl.X}.  Thus,  using that  $\mbox{dom}X=M$,   we conclude that $\mbox{dom}X^\epsilon=\mbox{dom}X$.
 The proof of the  last  part,  follows by combining  the  definition   in \eqref{enl.X} and  maximality  of $X$, and   by  taking into account  that $X\subset X^0$.
\end{proof}
\begin{proposition} \label{prop.elem.X}
Let  $X$, $X_1$ and $X_2$  be  multivalued monotone vector fields on $M$ and  $\epsilon, \epsilon_1, \epsilon_2 \geq 0.$ Then, there hold:
\begin{itemize}
\item[i)] If $\epsilon_1\geq\epsilon_2\geq0$ then $X^{\epsilon_2}\subset X^{\epsilon_1}$;
\item[ii)] $X_1^{\epsilon_1}+X_2^{\epsilon_2}\subset(X_1+X_2)^{\epsilon_1+\epsilon_2}$;
\item[iii)] $X^\epsilon(p)$ is closed and convex for all $p\in M$;
\item[iv)] $\alpha X^\epsilon=(\alpha X)^{\alpha\epsilon}$ for all $\alpha\geq0$;
\item[v)] $\alpha X^\epsilon_1+(1-\alpha)X^\epsilon_2\subset(\alpha X_1+(1-\alpha) X_2)^\epsilon$ for all $\alpha\in[0,1]$;
\item[vi)] If $E\subset\mathbb{R}_+$, then $\bigcap_{\epsilon\in E}X^\epsilon=X^{\overline{\epsilon}}$ with $\overline{\epsilon}=\inf E$.
\end{itemize}
\end{proposition}
\begin{proof} The proof is a consequence of  Definition~\ref{def.enl.X}  by using simple algebraic manipulations.
\end{proof}
\begin{proposition} \label{prop.conv.alg.}
Let   $X$  be a  multivalued monotone vector fields on $M$,  $\{ \epsilon^k\}$ be a sequence of positive numbers and $\{ (p^k, \, u^k)\} $ a sequence in $TM$.  If $\overline{\epsilon}=\lim_{k\rightarrow\infty}\epsilon^k$, $\overline{p}=\lim_{k\rightarrow\infty}p^k$, $\overline{u}=\lim_{k\rightarrow\infty}u^k$ and $u^k\in X^{\epsilon_k}(p^k)$ for all $k$, then $\overline{u}\in X^{\overline{\epsilon}}(\overline{p})$;
\end{proposition}
\begin{proof} Since $u^k\in X^{\epsilon_k}(p^k)$ for all $k$, then from  Definition \ref{def.enl.X} we have
$$
-\langle u^k, \, \exp_{p^k}^{-1}q \rangle+\langle-v,\exp_{q}^{-1}p^k\rangle  \geq  -\epsilon_k,  \qquad   q\in \mbox{dom}X, \quad  v\in X(q).
$$
Taking the limit in the last inequality,  as $k$ goes to $\infty$,  we obtain
$$
-\langle \overline{u}, \exp_{\bar{p}}^{-1}q \rangle+\langle-v,\exp_{q}^{-1}\bar{p}\rangle  \geq  -\overline{\epsilon},  \qquad   q\in \mbox{dom}X, \quad  v\in X(q).
$$
Therefore, using  again  Definition \ref{def.enl.X} the result  follows.
\end{proof}
\begin{proposition} \label{prop.loc.boun}
Suppose that $X$ is maximal monotone and $\emph{dom}X=M$. Then $X$ is locally bounded on M.
\end{proposition}
\begin{proof}
See  \cite[Lemma~{3.6}]{LiLopesMartin-Marquez2009}.
\end{proof}
\begin{proposition} \label{prop.boun.boun.}
If $X$ is maximal monotone and $\emph{dom}X=M$ then $X^\epsilon$ is bounded on bounded sets,  for all $\epsilon\geq0$.
\end{proposition}
\begin{proof}
Since  $X$ is  monotone and $\mbox{dom}X=M$, Proposition~\ref{prop.elem.ii} implies that  $\mbox{dom}X^\epsilon=M$.  Take  $V\subset M=\mbox{int}( \mbox{dom}X^\epsilon)$ a bounded set.  Note that   $\overline{V} \subset  \mbox{int}( \mbox{dom}X^\epsilon)$.  Let $r>0$ and define the set  $V_r=\{p\in M~:~d(p, V)\leq r\}$.  Taking into account that  $\mbox{dom}X=M$, then  $V_r\subset \mbox{dom}X$. Moreover,  since  both sets $V$ and $V_r$ are bounded,     Proposition~\ref{prop.loc.boun}  implies that   $m_X(V)< + \infty$  and $m_X(V_r)< + \infty$.   We are going to  prove that
\begin{equation} \label{eq:lxve}
m_{X^\epsilon}(V)\leq \frac{\epsilon}{r}+m_X(V_r)+2m_X(V).
\end{equation}
Take $p\in V$, $u\in X^\epsilon(p)$. Thus,  for all $v\in X(q)$, the definition of $X^\epsilon(p)$ in \eqref{enl.X} implies
$$
-\epsilon  \leq  -\langle u, \exp^{-1}_pq\rangle-\langle v, \exp^{-1}_qp\rangle.
$$
Let  $\hat{u}\in X(p)$. For $\hat{u} \neq u $ define $q=\exp_{p}w$, where $w=(r/\|u-\hat{u}\|)(u-\hat{u} )$. Thus, last inequality becomes
$$
-\epsilon \leq  -\|u-\hat{u}\| r-\langle \hat{u}, \exp^{-1}_pq\rangle-\langle v, \exp^{-1}_qp\rangle.
$$
Using that the  parallel transport is an isometry, we conclude from last inequality that
$$
-\epsilon \leq  - \|u-\hat{u}\| r+\|exp_{q}^{-1}p\|\|P_{qp}^{-1} \hat{u}-v\|.
$$
Since $r=\|exp_{q}^{-1}p\|$,  using  triangle inequality and once again that  the parallel transport is an isometry,   some manipulation in last inequality yields
$
\|u-\hat{u}\|  \leq   \epsilon/r+ \|\hat{u}\|+\|v\|.
$
Hence, taking into account that $\|u\|\leq\|u-\hat{u}\| + \|\hat{u}\|$,  we  obtain
$$
\|u\|\leq    \frac{\epsilon} {r}+ 2 \|\hat{u}\|+\|v\|.
$$
Note that last inequality also holds for $u=\hat{u}$. Since $\|exp_{q}^{-1}p\|=r$ and $p\in V$, we have $q\in V_r$.  Thus,  $\|\hat{u}\|\leq m_X(\Omega)$ and $\|v\|\leq m_X(\Omega_r)$, which imples that
$$
\|u\|\leq \frac{\epsilon} {r}+m_X(\Omega_r)+2m_X(\Omega).
$$
Since $u$ is an arbitrary element of $X^\epsilon(\Omega)$, the inequality in \eqref{eq:lxve} follows, and the proof is concluded.
\end{proof}
%%%%%%%%%%%%%%%%%%%%%%%%%%%%%%%%%%%%%%
\section{ An Inexact Proximal Point Method for Variational Inequalities} \label{sec3}
%%%%%%%%%%%%%%%%%%%%%%%%%%%%%%%%%%%%%%
Let $X:M\rightrightarrows TM$ be a multivalued vector field and $ \Omega \subset M$ be a nonempty set.  The {\it variational inequality problem} VIP(X,$\Omega$) consists of finding $p^*\in\Omega$ such that there exists $u\in X(p^*)$ satisfying
$$
\langle u,\,\exp^{-1}_{p^*}q\rangle\geq0, \qquad   q\in\Omega.
$$
Using \eqref{eq:nc}, i.e., the definition of normal  cone to $\Omega$, the VIP(X,$\Omega$) becomes the problem of finding $p^*\in\Omega$ satisfying the inclusion
\begin{equation} \label{eq.vip}
0\in X(p)+N_{\Omega}(p).
\end{equation}
\begin{remark}
In particular, if  $\Omega=M$, then  $N_{\Omega}(p)=\{0\}$ and  \emph{VIP}(X,$\Omega$) becomes  to the problem of   finding $p^*\in\Omega$ such that $0\in X(p^*).$
\end{remark}
From now on $S(Y,\,\Omega)$  denotes the  solution set of the inclusion  (\ref{eq.vip}). We need  of the following   three assumptions:
\begin{itemize}
\item[{\bf A1.}] $Y:=X+N_{\Omega}$ with $\mbox{dom} X=M$ and $\Omega$ closed and convex;
\item[ {\bf A2.}] $X$ is maximal monotone;
\item[ {\bf A3.}] $S(X,\,\Omega)\neq \varnothing$.
\end{itemize}
Take   $0<\hat{\lambda}\leq\tilde{\lambda}$,  a sequence $\{\lambda_k\}\subset\mathbb{R}$ such that $\hat{\lambda}\leq\lambda_k\leq\tilde{\lambda}$ and a sequence $\{\epsilon_k\}\subset\mathbb{\mathbb{R}_{++}}$ such that $\sum_{k=0}^{\infty}\epsilon_k<\infty$.
The {\it proximal point method} for VIP(X,\,$\Omega$) is defined as follows:  Given $p^0\in\Omega$  take $p^{k+1}$ such that
\begin{equation}  \label{eq.pk+1ii}
0\in (X^{\epsilon_k}+N_{\Omega})(p^{k+1})-2\lambda_k\exp^{-1}_{p^{k+1}}p^k, \qquad k=0, 1\ldots .
\end{equation}
\begin{remark}
The method \eqref{eq.pk+1ii} has many important particular instances. For example, in the case $\epsilon_k=0$ for all $k$, we obtain the method (5.1) of \cite{LiYao2012}. For  $\Omega=M$ and $\epsilon_k=0$ for all $k$, we obtain the method (4.3) of \cite{LiLopesMartin-Marquez2009}.  For  $M=\mathbb{R}^n$, we obtain the method (23)-(25) of \cite{BurachikIusemSvaiter1997}, where the Bregman distance  is induced by the square of the Euclidean norm and $C=\mathbb{R}^n$.
\end{remark}

\begin{lemma} \label{le:esvi}
  For each $q\in M$ and $\lambda > 0$ the following inclusion problem
$$
0\in X(p)-2\lambda \exp^{-1}_p q+N_{\Omega}(p), \quad p\in M.
$$
has an unique solution.
\end{lemma}
\begin{proof}
Since $X$ is a monotone vector field and $\lambda > 0$,  combining Proposition~\ref{prop.sm}  with \eqref{eq:gd2},  we conclude that the vector field $ Z(p)=X(p)-2\lambda \exp^{-1}_p q$ is a strongly maximal monotone vector field. Therefore, using that $Z$ is maximal  and  taking into account that $M$ is a Hadamard manifold and $ \Omega $  is a nonempty and convex set, we may combine \cite[Proposition 3.5]{LiLopesMartin-Marquez2009} with    \cite[Corollary 3.14]{LiYao2012} to conclude the proof.
\end{proof}
Now we are going to prove the convergence result for the  proximal point method \eqref{eq.pk+1ii}.
\begin{theorem}\label{conv.alg.ii} Assume that {\bf A1}-{\bf A3} hold. Then, the sequence $\{p^k\}$ generated by \eqref{eq.pk+1ii}   is well defined and  converges to a point   $p^*\in S(X,\,\Omega)$.
\end{theorem}
\begin{proof}
Since $\mbox{dom} X=M$, Proposition~\ref{prop.elem.ii} and item i of Proposition~\ref{prop.boun.boun.} imply that $ X(p)\subseteq X^{\epsilon_k}(p)$ for all $p\in M$ and $k=0,1, \ldots$.  Hence,  for proving the well definition of the sequence $\{p^k\}$ it is sufficient to prove that the inclusion
$$
0\in X(p)-2\lambda_k\exp^{-1}_pp^k+N_{\Omega}(p), \qquad p\in M, 
$$
has solution,   for each $k=0,1, \ldots$, which is a consequence of Lemma~\ref{le:esvi}.

Now, we are going to prove the convergence of $\{p^k\}$ to a point   $p^*\in S(X,\,\Omega)$. Using Proposition~\ref{prop.elem.ii} we conclude that
$N_{\Omega}\subset N_{\Omega}^0$. Thus, from item ii of Proposition~\ref{prop.elem.X} we have  $X^{\epsilon_k}+N_{\Omega} \subset (X+N_{\Omega})^{\epsilon_k}$, for all  $k=0,1, \ldots$. Therefore, using \eqref{eq.pk+1ii} we  obtain
\begin{equation} \label{eq:icte}
2\lambda_k\exp^{-1}_{p^{k+1}}p^k\in (X+N_{\Omega})^{\epsilon_k}(p^{k+1}), \qquad k=0,1, \ldots.
\end{equation}
Since $P_{qp^{k+1}}^{-1} \exp^{-1}_{q}p^{k+1}=-\exp^{-1}_{p^{k+1}}q$ and the parallel transport is a isometry,  last inclusion together  with  Definition~\ref{def.enl.X} yield
$$
-2\lambda_k\left\langle\exp^{-1}_{p^{k+1}}p^k,\exp^{-1}_{p^{k+1}}q\right\rangle+\left\langle v,-\exp^{-1}_qp^{k+1}\right\rangle  \geq  -\epsilon_k,\quad  q\in\Omega, ~  v\in(X+N_{\Omega})(q),  \quad k=0,1, \ldots.
$$
Particularly, if $q\in S(X,\,\Omega)$ then $0\in X+N_{\Omega}(q)$ and last inequality becomes
\begin{eqnarray*}
-2\lambda_k\left\langle\exp^{-1}_{p^{k+1}}p^k,\exp^{-1}_{p^{k+1}}q\right\rangle\geq-\epsilon_k, \quad q\in S(X,\,\Omega),   \qquad k=0,1, \ldots.
\end{eqnarray*}
Using  last inequality and \eqref{eq:coslaw} with $p_1=p^k$, $p_2=q$ and $p_3=p^{k+1}$, after some  algebras   we obtain
\begin{equation}\label{eq.the.ii.X}
-\frac{\epsilon_k}{2\lambda_k}\leq d^2(q,p^k)-d^2(p^k, p^{k+1})-d^2(q,p^{k+1}), \qquad q\in S(X,\,\Omega),   \qquad k=0,1, \ldots.
\end{equation}
Since $0<\hat{\lambda}\leq\lambda_k$, the last inequality gives
\begin{equation} \label{eq.fejer.X}
d^2(q, p^{k+1})\leq d^2(q, p^k)+\frac{\epsilon_k}{\hat{\lambda}}, \qquad q\in S(X,\,\Omega),   \qquad k=0,1, \ldots.
\end{equation}
 Because  $\sum_{k=0}^{\infty}\epsilon_k<\infty$ and $S(X,\,\Omega)\neq \varnothing$, last inequality implies that $\{p^k\}$ is   quasi-Fej\'er convergent to $S(X,\,\Omega)$.  From Proposition~\ref{fejer},  for concluding the proof    is sufficient to prove that there exists an accumulation point $\bar{p}$ of $\{p^k\}$ belonging  to $S(X,\,\Omega)$.  Since  $\{p^k\}$ is   quasi-Fej\'er convergent to $S(X,\,\Omega)$, Proposition~\ref{fejer} implies that  $\{p^k\}$ is bounded.  Take $\bar{p}$ and  $\{p^{n_k}\}$    an accumulation point  and   a subsequence of  $\{p^k\}$, respectively,  such that  $\bar{p}=\lim_{k\rightarrow\infty}p^{n_k}$. On the other hand, since $0<\hat{\lambda}\leq\lambda_k$ and $\sum_{k=0}^{\infty}\epsilon_k<\infty$,  the inequality in \eqref{eq.the.ii.X} implies  that $\lim_{k\to\infty} d(p^k, p^{k+1})=0$. Thus,  $ \lim_{k\to\infty} \exp^{-1}_{p^{n_k+1}}p^{n_k}=0$ and   $\lim_{k\rightarrow\infty}p^{n_k+1}=\bar{p}$. Now, using \eqref{eq:icte} we have
$$
2\lambda_{n_k}\exp^{-1}_{p^{n_k+1}}p^{n_k}\in (X+N_{\Omega})^{\epsilon_{n_k}}(p^{n_k+1}),    \qquad k=0,1, \ldots.
$$
Therefore, letting $k$ goes to $\infty$ in the last inclusion and using  Proposition \ref{prop.conv.alg.}, Lemma \ref{mon.cone}, Proposition \ref{prop.elem.ii} and taking into account that $\{\lambda_k\}$ is bounded we obtain
\begin{eqnarray*} 0\in(X+N_{\Omega})(\bar{p}),
\end{eqnarray*}
which implies that  $\bar{p}\in S(X,\,\Omega)$ and the proof is concluded.
\end{proof}
%%%%%%%%%%%%%%%%%%%%%%%%%%%%%%%%%%%%%%
\section{An Inexact Proximal Point Method for Otimization} \label{sec4}
%%%%%%%%%%%%%%%%%%%%%%%%%%%%%%%%%%%%%%
Throughout this section,  we assume that   $f:M \rightarrow \mathbb{R}$  is a  convex function.  The enlargement of the subdifferential of $f$,   denoted  by  $\partial^{\epsilon}f: M   \rightrightarrows  TM $,   is defined by
$$
    \partial^{\epsilon} f(p):=\left\{ u\in T_pM~:~  \left\langle \mbox{P}_{qp}^{-1} u-v, \,  \exp_{q}^{-1}p\right\rangle \geq -\epsilon, ~ q\in M, ~  v\in \partial f(q) \right\},  \qquad \epsilon\geq 0.
$$
and we denote the {\it $\epsilon$-subdifferential} of $f$  by $\partial_{\epsilon}f : M   \rightrightarrows  TM$, which  is given by
$$
\partial_{\epsilon}f(p):=\left\{u\in T_pM ~ : ~ f(q) \geq f(p)+\langle u,\exp^{-1}_pq\rangle - \epsilon, ~  q\in M\right\}, \qquad \epsilon\geq 0.
$$
\begin{example}\label{ex:xesg}  Let  $\epsilon\geq 0$ and   $\bar{p}\in M$.   Define the closed ball at the origin $0_{T_pM}$ of $T_pM$ and radius $2\sqrt{\epsilon}$ by 
 $$
 B\left[0_{T_pM}, ~2\sqrt{\epsilon}\right]:=\left\{w\in T_pM~:~\parallel w\parallel \leq 2\sqrt{\epsilon}\right\}.
 $$
 Denote the  $\epsilon$-subdifferential  of   $\partial d_{\bar{p}}^2(p)=\{ \grad d_{\bar{p}}^2(p)\}$ by   $\partial_{\epsilon} d_{\bar{p}}^2$.   We claim that the following inclusion holds 
 $$
 \partial d_{\bar{p}}^2(p) + B\left[0_{T_pM}, ~2\sqrt{\epsilon}\right] \subseteq \partial_{\epsilon} d_{\bar{p}}^2(p),  \qquad   p\in M.
 $$
 Indeed,  first note that   from  \eqref{eq:gd2}  we conclude that   $\partial d_{\bar{p}}^2(q)=\{-2\exp^{-1}_q\bar{p}\}$,  for each $ q\in M$.  Due to $ \emph{dom} \partial d_{\bar{p}}^2=M$   definition of $ \partial_{\epsilon} d_{\bar{p}}^2$ implies 
  \begin{equation}\label{ex.1.isg}
  \partial_{\epsilon} d_{\bar{p}}^2(p)=\left\{ u\in T_pM~:~ d^2(\bar{p}, q) \geq d^2(\bar{p}, p)+\langle u,\exp^{-1}_pq\rangle -\epsilon, ~  q\in M \right\}, \quad p\in M.
\end{equation}
We are going to prove the auxiliary result    $\{-2\exp^{-1}_p\bar{p}\}+A(p)\subset  \partial_{\epsilon} d_{\bar{p}}^2(p)$  for each $ p\in M$, where  \begin{equation}\label{eq:apsg}
B(p)=\left\{w\in T_pM ~:~ 0 \geq -d^2(p,q)+ \|w\|d(p,q)- \epsilon, ~ q\in M \right\}, \quad p\in M.
\end{equation}
First of all, note that by  using  \eqref{eq:coslaw},  we obtain the following inequality 
$$
d^2(\bar{p}, q)-d^2(\bar{p}, p)-d^2(p,q)+2\langle \exp^{-1}_p\bar{p},\exp^{-1}_pq\rangle \geq 0,  \quad p, q \in M.
$$
Take  $w\in B(p)$.  Since $\langle w,exp^{-1}_pq\rangle\leq \|w\|d(p,q)$,  for all $w\in B(p)$ and $p, q\in M$, combining   \eqref{eq:apsg} with last inequality yields 
$$
d^2(\bar{p}, q)-d^2(\bar{p}, p)-d^2(p,q)+2\langle \exp^{-1}_p\bar{p},\exp^{-1}_pq\rangle  \geq -d^2(p,q)+\langle w,exp^{-1}_pq\rangle-\epsilon, \quad p, q \in M.
$$
Simple algebraic manipulations  in last inequality shows  that it is equivalent to the following ones
$$
d^2(\bar{p}, q) \geq d^2(\bar{p}, p)+\langle -2\exp^{-1}_p\bar{p}+ w,\exp^{-1}_pq\rangle -\epsilon, \quad p, q \in M, 
$$
which, from \eqref{ex.1.isg},  allows to conclude that $-2\exp^{-1}_p\bar{p} +w \in  \partial_{\epsilon} d_{\bar{p}}^2(p)$, for all $w\in B(p)$ and $p\in M$.  Thus, the auxiliary result is proved. Finally, note that $w\in B(p)$ if,  and only if, there holds $\|w\|^2-4\epsilon <0, $ or equivalently, $\|w\|<2\sqrt{\epsilon}$. Therefore, $B(p)= B\left[0_{T_pM}, ~2\sqrt{\epsilon}\right]$ and,  because  $ \partial d_{\bar{p}}^2(p)+A(p)\subset  \partial_{\epsilon} d_{\bar{p}}^2(p)$  for each $ p\in M$,    the proof of the claim is done.  
\end{example}
\begin{proposition} \label{prop.epsilon}
For each  $p\in M$, there holds $\partial_\epsilon f(p)\subseteq \partial^\epsilon f(p)$.
\end{proposition}
\begin{proof}
Take $u\in\partial_\epsilon f(p)$,  $q\in M$ and  $v\in \partial f(q)$.  From  the definitions  of $\partial f(q)$ and $\partial_\epsilon f (p)$ we have
$$
f(p) \geq f(q)+\langle v,\exp^{-1}_qp\rangle, \qquad  \qquad f(q) \geq f(p)+\langle u,\exp^{-1}_pq\rangle - \epsilon,
$$
respectively. Combining two last inequalities  we conclude that  $ 0\geq  \langle v, \exp^{-1}_qp\rangle+ \langle u,\exp^{-1}_pq\rangle +\epsilon. $
Since  the parallel transport is an isometry and $\mbox{P}_{qp}^{-1} \exp^{-1}_pq = -\exp^{-1}_qp $, last inequality becomes
$$
0\geq  \langle v, \, \exp^{-1}_qp\rangle+ \langle \mbox{P}_{qp}^{-1}u, \, -\exp^{-1}_qp\rangle -\epsilon.
$$
Thus, using  last inequality and  definition of $ \partial^{\epsilon} f(p)$ we obtain  that $u\in \partial^\epsilon f(p)$. Therefore,  the prove is done.
\end{proof}
\begin{remark}
Note that if $M$ has zero curvature then the inequality \eqref{eq:coslaw} holds as a equality. Therefore,  
 in Example~\ref{ex:xesg}, we can prove that  the equality holds as equality, namely, 
$$
 \partial d_{\bar{p}}^2(p) + B\left[0_{T_pM}, ~2\sqrt{\epsilon}\right] = \partial_{\epsilon} d_{\bar{p}}^2(p),  \qquad   p\in M. 
 $$
 Moreover, we can also prove that the inclusion  $ \partial_{\epsilon} d_{\bar{p}}^2(p)\subset  \partial^{\epsilon} d_{\bar{p}}^2(p)$ is strict,   for all   $p\in M$, see Example~\ref{ex:xe}.
\end{remark}
Let $\Omega \subset M$. The {\it constrained optimization problem}  consists in
\begin{equation}  \label{eq.cop}
 \min  ~ f(p) , \qquad  \const  ~ p\in\Omega.
\end{equation}
Letting $\delta_\Omega$ be the indicate function,  defined by $\delta_{\Omega}(p)=0$, if $p\in\Omega$ and $\delta_{\Omega}(p)=+\infty$ otherwise,   Problem~\ref{eq.cop} is equivalent to
\begin{eqnarray*}
 \min  ~ (f+\delta_{\Omega})(p), \qquad  \const  ~ p\in M.
\end{eqnarray*}
From now on,   $\Omega \subset M$  is  a closed and convex set  and $S(f,\Omega)$ denotes  the solution  set of  Problem~\ref{eq.cop}.
\begin{theorem} \label{th:copoc}
There holds   $ \partial (f+\delta_{\Omega})(p)=\partial f(p)+N_{\Omega}(p)$, for each $p\in\Omega.$  Moreover,   $p^*\in S(f,\Omega)$ if,  and  only if,    $0\in \partial f(p^*)+N_{\Omega}(p^*)$.
\end{theorem}
\begin{proof}  The first part was proved in \cite[Proposition 5.4]{LiLopesMartin-Marquez2009}.  To prove the second part, first  use convexity of  $\Omega$ and $f$  for concluding  that  $f+\delta_{\Omega}$  is also convex,  and then use  the first part to obtain the result.
\end{proof}

Take   $0<\hat{\lambda}\leq\tilde{\lambda}$,  a sequence $\{\lambda_k\}\subset\mathbb{R}$ such that $\hat{\lambda}\leq\lambda_k\leq\tilde{\lambda}$ and a sequence $\{\epsilon_k\}\subset\mathbb{\mathbb{R}_{++}}$ such that $\sum_{k=0}^{\infty}\epsilon_k<\infty$. The {\it inexact proximal point method  for the constrained optimization problem} in \eqref{eq.cop}  is defined as follows: \\
\emph{Inicialization:}
\begin{equation} \label{eq.ia.iif}
p^0\in\Omega.
\end{equation}
\emph{Iterative Step:} Given $p^k$, define $X_k:M\rightrightarrows TM$ as
\begin{equation} \label{eq.pia.iif}
X_k(p):=(\partial^{\epsilon_k}f+N_{\Omega})(p)-2\lambda_k\exp^{-1}_px^k,
\end{equation}
and take $p^{k+1}$ such that
\begin{equation} \label{eq.pk+1iif}
0\in X_k(p^{k+1}).
\end{equation}
\begin{remark}  For  $\epsilon_k=0$ the above method  generalizes  the method (5.15) of  Chong Li et. al. \cite{LiLopesMartin-Marquez2009} and,   for $\epsilon_k=0$ and  $\Omega=M$ we obtain the method proposed by  Ferreira and Oliveira \cite{FerreiraOliveira2002}.
\end{remark}
\begin{theorem} \label{conv.alg.}
Assume that $S(f,\,\Omega)\neq \varnothing$. Then, the sequence $\{p^k\}$ generated by \eqref{eq.ia.iif}-\eqref{eq.pk+1iif}   is well defined and  converges to a point   $p^*\in S(f,\,\Omega)$.
\end{theorem}
\begin{proof}    Since $\mbox{dom}f=M$,   Theorem~\ref{mmsub} implies that  $\partial f$ is maximal monotone. Therefore, taking into account  that  $N_{\Omega}= \partial \delta_{\Omega}$,  the result follows  directly from Theorem ~\ref{conv.alg.ii} with $X=\partial f$.
\end{proof}
%%%%%%%%%%%%%%%%%%%%%%%%%%%%%%%%%%%
\section{Final Remarks} \label{secfr}
%%%%%%%%%%%%%%%%%%%%%%%%%%%%%%%%%%%
In this paper we study some basics properties of enlargement of monotone vector fields. Since this concept has been successfully employed  for wide range of purpose, in linear setting,  we expect that the results of this paper become a first step towards a more general theory in the Riemannian context, including other algorithms for solving variational inequalities. We foresee further progress in this topic in the nearby future.
%%%%%%%%%%%%%%%%%%%
\section*{Acknowledgements}
%%%%%%%%%%%%%%%%%%%
The work  was supported  by CNPq Grants   458479/2014-4,  471815/2012-8,  303732/2011-3, 312077/2014-9,  305158/2014-7,  CAPES-MES-CUBA 226/2012 and  FAPEG  201210267000909 - 05/2012.


\begin{thebibliography}{10}

\bibitem{Absil2014}
Absil, P.A., Amodei, L., Meyer, G.: Two {N}ewton methods on the manifold of
  fixed-rank matrices endowed with {R}iemannian quotient geometries.
\newblock Comput. Statist. \textbf{29}(3-4), 569--590 (2014)

\bibitem{AdlerDedieuShub2002}
Adler, R.L., Dedieu, J.P., Margulies, J.Y., Martens, M., Shub, M.: Newton's
  method on {R}iemannian manifolds and a geometric model for the human spine.
\newblock IMA J. Numer. Anal. \textbf{22}(3), 359--390 (2002)


\bibitem{AhmadiKhatibzadeh2014}
Ahmadi, P., Khatibzadeh, H.: On the convergence of inexact proximal point
  algorithm on {H}adamard manifolds.
\newblock Taiwanese J. Math. \textbf{18}(2), 419--433 (2014)

\bibitem{AABCP2014-3}
Amat, S., Busquier, S., Castro, R., Plaza, S.: Third-order methods on
  {R}iemannian manifolds under {K}antorovich conditions.
\newblock J. Comput. Appl. Math. \textbf{255}, 106--121 (2014)


\bibitem{Bacak2013}
Ba{\v{c}}{\'a}k, M.: The proximal point algorithm in metric spaces.
\newblock Israel J. Math. \textbf{194}(2), 689--701 (2013)


\bibitem{BauschkeCombettes2011}
Bauschke, H.H., Combettes, P.L.: Convex analysis and monotone operator theory
  in {H}ilbert spaces. 
  \newblock CMS Books in Mathematics/Ouvrages de Math\'ematiques de la SMC.
  Springer, New York, 2011.


\bibitem{BentoCruzNeto2013}
Bento, G.C., Cruz~Neto, J.X.: A subgradient method for multiobjective
  optimization on {R}iemannian manifolds.
\newblock J. Optim. Theory Appl. \textbf{159}(1), 125--137 (2013)

\bibitem{BCNS2013}
Bento, G.C., da~Cruz~Neto, J.X., Santos, P.S.M.: An inexact steepest descent
  method for multicriteria optimization on {R}iemannian manifolds.
\newblock J. Optim. Theory Appl. \textbf{159}(1), 108--124 (2013)


\bibitem{BFO2012}
Bento, G.C., Ferreira, O.P., Oliveira, P.R.: Unconstrained steepest descent
  method for multicriteria optimization on {R}iemannian manifolds.
\newblock J. Optim. Theory Appl. \textbf{154}(1), 88--107 (2012)


\bibitem{BentoFerreira2015}
Bento, G.C., Ferreira, O.P., Oliveira, P.R.: Proximal point method for a
  special class of nonconvex functions on {H}adamard manifolds.
\newblock Optimization \textbf{64}(2), 289--319 (2015)


\bibitem{BentoMelo2012}
Bento, G.C., Melo, J.G.: Subgradient method for convex feasibility on
  {R}iemannian manifolds.
\newblock J. Optim. Theory Appl. \textbf{152}(3), 773--785 (2012)

\bibitem{BonnelUd2015}
Bonnel, H., Todjihounde, L., Udriste, C.: Semivectorial bilevel optimization on
  riemannian manifolds.
\newblock Journal of Optimization Theory and Applications. \textbf{0}(0), 1--23
  (2015)

\bibitem{BurachikDIS1995}
Burachik, R., Drummond, L.M.G., Iusem, A.N., Svaiter, B.F.: Full convergence of
  the steepest descent method with inexact line searches.
\newblock Optimization \textbf{32}(2), 137--146 (1995)


\bibitem{BurachikIusem1998}
Burachik, R.S., Iusem, A.N.: A generalized proximal point algorithm for the
  variational inequality problem in a {H}ilbert space.
\newblock SIAM J. Optim. \textbf{8}(1), 197--216 (1998)

\bibitem{BurachikIusem2008}
Burachik, R.S., Iusem, A.N.: Set-valued mappings and enlargements of monotone
  operators, \emph{Springer Optimization and Its Applications}, vol.~8.
\newblock Springer, New York, 2008.

\bibitem{BurachikIusemSvaiter1997}
Burachik, R.S., Iusem, A.N., Svaiter, B.F.: Enlargement of monotone operators
  with applications to variational inequalities.
\newblock Set-Valued Anal. \textbf{5}(2), 159--180 (1997)


\bibitem{doCarmo1992}
do~Carmo, M.P.: Riemannian geometry.
\newblock Mathematics: Theory \& Applications. Birkh\"auser Boston, Inc.,
  Boston, MA, 1992.


\bibitem{Cioranescu1990}
Cioranescu, I.: Geometry of {B}anach spaces, duality mappings and nonlinear
  problems, \emph{Mathematics and its Applications}, vol.~62.
\newblock Kluwer Academic Publishers Group, Dordrecht, 1990.


\bibitem{CLMM2012}
Colao, V., L{\'o}pez, G., Marino, G., Mart{\'{\i}}n-M{\'a}rquez, V.:
  Equilibrium problems in {H}adamard manifolds.
\newblock J. Math. Anal. Appl. \textbf{388}(1), 61--77 (2012)


\bibitem{daCruzFerreiraPerez2002}
da~Cruz~Neto, J.X., Ferreira, O.P., Lucambio~P{\'e}rez, L.R.: Contributions to
  the study of monotone vector fields.
\newblock Acta Math. Hungar. \textbf{94}(4), 307--320 (2002)


\bibitem{FerreiraCPN2006}
Da~Cruz~Neto, J.X., Ferreira, O.P., P{\'e}rez, L.R.L., N{\'e}meth, S.Z.:
  Convex- and monotone-transformable mathematical programming problems and a
  proximal-like point method.
\newblock J. Global Optim. \textbf{35}(1), 53--69 (2006)


\bibitem{daCruzFerreiraPerez2006}
Da~Cruz~Neto, J.X., Ferreira, O.P., P{\'e}rez, L.R.L., N{\'e}meth, S.Z.:
  Convex- and monotone-transformable mathematical programming problems and a
  proximal-like point method.
\newblock J. Global Optim. \textbf{35}(1), 53--69 (2006)

\bibitem{EdelmanAriasSmith1999}
Edelman, A., Arias, T.A., Smith, S.T.: The geometry of algorithms with
  orthogonality constraints.
\newblock SIAM J. Matrix Anal. Appl. \textbf{20}(2), 303--353 (1999)

\bibitem{FangChen2015}
Fang, C.j., Chen, S.l.: A projection algorithm for set-valued variational
  inequalities on {H}adamard manifolds.
\newblock Optim. Lett. \textbf{9}(4), 779--794 (2015)


\bibitem{FerreiraOliveira2002}
Ferreira, O.P., Oliveira, P.R.: Proximal point algorithm on {R}iemannian
  manifolds.
\newblock Optimization \textbf{51}(2), 257--270 (2002)


\bibitem{HosseiniPouryayevali2013}
Hosseini, S., Pouryayevali, M.R.: Nonsmooth optimization techniques on
  {R}iemannian manifolds.
\newblock J. Optim. Theory Appl. \textbf{158}(2), 328--342 (2013)


\bibitem{Karmarkar1998}
Karmarkar, N.: Riemannian geometry underlying interior-point methods for linear
  programming.
\newblock In: Mathematical developments arising from linear programming
  ({B}runswick, {ME}, 1988), \emph{Contemp. Math.}, vol. 114, pp. 51--75. Amer.
  Math. Soc., Providence, RI (1990)


\bibitem{LiLopesMartin-Marquez2009}
Li, C., L{\'o}pez, G., Mart{\'{\i}}n-M{\'a}rquez, V.: Monotone vector fields
  and the proximal point algorithm on {H}adamard manifolds.
\newblock J. Lond. Math. Soc. (2) \textbf{79}(3), 663--683 (2009)


\bibitem{LiYao2012}
Li, C., Yao, J.C.: Variational inequalities for set-valued vector fields on
  {R}iemannian manifolds: convexity of the solution set and the proximal point
  algorithm.
\newblock SIAM J. Control Optim. \textbf{50}(4), 2486--2514 (2012)


\bibitem{LiChongLiouYao2009}
Li, S.L., Li, C., Liou, Y.C., Yao, J.C.: Existence of solutions for variational
  inequalities on {R}iemannian manifolds.
\newblock Nonlinear Anal. \textbf{71}(11), 5695--5706 (2009)


\bibitem{Manton2015}
Manton, J.H.: A framework for generalising the {N}ewton method and other
  iterative methods from {E}uclidean space to manifolds.
\newblock Numer. Math. \textbf{129}(1), 91--125 (2015)


\bibitem{MillerMalick2005}
Miller, S.A., Malick, J.: Newton methods for nonsmooth convex minimization:
  connections among {U}-{L}agrangian, {R}iemannian {N}ewton and {SQP} methods.
\newblock Math. Program. \textbf{104}(2-3, Ser. B), 609--633 (2005)

\bibitem{Nemeth2003}
N{\'e}meth, S.Z.: Variational inequalities on {H}adamard manifolds.
\newblock Nonlinear Anal. \textbf{52}(5), 1491--1498 (2003)

\bibitem{NesterovTodd2002}
Nesterov, Y.E., Todd, M.J.: On the {R}iemannian geometry defined by
  self-concordant barriers and interior-point methods.
\newblock Found. Comput. Math. \textbf{2}(4), 333--361 (2002)


\bibitem{PapaQuirozOliveira2012}
Papa~Quiroz, E.A., Oliveira, P.R.: Full convergence of the proximal point
  method for quasiconvex functions on {H}adamard manifolds.
\newblock ESAIM Control Optim. Calc. Var. \textbf{18}(2), 483--500 (2012)


\bibitem{Rapcsak1997}
Rapcs{\'a}k, T.: Smooth nonlinear optimization in {$\bold R^n$},
  \emph{Nonconvex Optimization and its Applications}, vol.~19.
\newblock Kluwer Academic Publishers, Dordrecht, 1997.


\bibitem{Sakai1996}
Sakai, T.: Riemannian geometry, \emph{Translations of Mathematical Monographs},
  vol. 149.
\newblock American Mathematical Society, Providence, RI, 1996.


\bibitem{Smith1994}
Smith, S.T.: Optimization techniques on {R}iemannian manifolds.
\newblock In: Hamiltonian and gradient flows, algorithms and control,
  \emph{Fields Inst. Commun.}, vol.~3, pp. 113--136. Amer. Math. Soc.,
  Providence, RI (1994)

\bibitem{TangWangLiu2015}
Tang, G.j., Wang, X., Liu, H.w.: A projection-type method for variational
  inequalities on {H}adamard manifolds and verification of solution existence.
\newblock Optimization \textbf{64}(5), 1081--1096 (2015)

\bibitem{TangZhouHuang2013}
Tang, G.j., Zhou, L.w., Huang, N.j.: The proximal point algorithm for
  pseudomonotone variational inequalities on {H}adamard manifolds.
\newblock Optim. Lett. \textbf{7}(4), 779--790 (2013)


\bibitem{WangLi2015}
Wang, J., Li, C., Lopez, G., Yao, J.C.: Convergence analysis of inexact
  proximal point algorithms on {H}adamard manifolds.
\newblock J. Global Optim. \textbf{61}(3), 553--573 (2015)


\bibitem{WangLiLopezYao2015}
Wang, J., Li, C., Lopez, G., Yao, J.C.: Convergence analysis of inexact
  proximal point algorithms on {H}adamard manifolds.
\newblock J. Global Optim. \textbf{61}(3), 553--573 (2015)


\bibitem{WangLiYao2015}
Wang, X.M., Li, C., Yao, J.C.: Subgradient projection algorithms for convex
  feasibility on {R}iemannian manifolds with lower bounded curvatures.
\newblock J. Optim. Theory Appl. \textbf{164}(1), 202--217 (2015)


\bibitem{WenWotao2013}
Wen, Z., Yin, W.: A feasible method for optimization with orthogonality
  constraints.
\newblock Math. Program. \textbf{142}(1-2, Ser. A), 397--434 (2013)


\end{thebibliography}
\end{document}